\documentclass{amsart}
\usepackage{amsmath}
\usepackage{amssymb}
\usepackage{amsthm}
\usepackage{hyperref}

\newtheorem{theorem}{Theorem}[section]
\newtheorem{lemma}[theorem]{Lemma}
\newtheorem{proposition}[theorem]{Proposition}
\newtheorem{corollary}[theorem]{Corollary}

\theoremstyle{definition}
\newtheorem{definition}[theorem]{Definition}

\theoremstyle{remark}
\newtheorem{remark}[theorem]{Remark}

\numberwithin{equation}{section}

\DeclareMathOperator{\Lip}{\mathrm{Lip}}

\begin{document}

\title[Hamilton-Jacobi equations on metric spaces]{Stability properties and large time behavior of viscosity solutions of Hamilton-Jacobi equations on metric spaces}

\author{Atsushi~Nakayasu}
\address{Graduate School of Mathematical Sciences, The University of Tokyo,
3-8-1 Komaba, Meguro, Tokyo, 153-8914 Japan}
\curraddr{}
\email{ankys@ms.u-tokyo.ac.jp}
\thanks{}

\author{Tokinaga~Namba}
\address{Graduate School of Mathematical Sciences, The University of Tokyo,
3-8-1 Komaba, Meguro, Tokyo, 153-8914 Japan}
\curraddr{}
\email{namba@ms.u-tokyo.ac.jp}
\thanks{}



\date{\today}



\begin{abstract}
We investigate asymptotic behaviors of a metric viscosity solution of a Hamilton-Jacobi equation defined on a general metric space in Gangbo-\'{S}wi\c{e}ch sense.
Our results include general stability and large time behavior of the solution.
\end{abstract}

\maketitle


\section{Introduction}
\label{s:infto}

In this chapter we study stability of a solution of Hamilton-Jacobi equations on a general complete geodesic metric space space $(X, d)$.
Let $H$ be a continuous function on $X\times\mathbf{R}_+$ called a \emph{Hamiltonian}.
Consider a Cauchy problem of a Hamilton-Jacobi equation of the form
\begin{align}
\label{e:hj}
\partial_t u+H(x, |\mathit{D}u|) = 0 &\quad \text{in $(0, \infty)\times X$,} \\
\label{e:hj}
u|_{t = 0} = u_0 &
\end{align}
with initial data $u_0$
and a corresponding stationary equation of the form
\begin{equation}
\label{e:lhj}
H(x, |\mathit{D}v|) = c \quad \text{in $X$}
\end{equation}
with some $c \in \mathbf{R}$.

The theory of Hamilton-Jacobi equations on generalized spaces have been developing in these years.
For example, \cite{SC11} and \cite{CCM15} study a stationary equations on topological networks and post-critically finite fractals including the Sierpinski gasket.
In order to cover them a metric viscosity solution is posed,
which means a theory of viscosity solutions on a general metric space.
A notion of metric viscosity solution was first introduced by Giga-Hamamuki-Nakayasu \cite{GHN14} to the stationary equation \eqref{e:lhj} in a spirit of \cite{SC11}.
It was attempted to apply this idea to the evolutionary equation \eqref{e:hj} in \cite{N14}.
Afterwards some different notions of metric viscosity solutions were proposed by several authors;
see, e.g., \cite{AF14}, \cite{GS14} \cite{GS14b}.
In particular, the metric viscosity solution by Gangbo-\'{S}wi\c{e}ch \cite{GS14}, \cite{GS14b} is apparently compatible with stability argument.
In fact, the authors of \cite{GS14} construct a solution of \eqref{e:hj} by Perron method
while the other materials show a representation formula of a metric viscosity solution.

The main aim of the present work is to establish a general stability result for the Gangbo-\'{S}wi\c{e}ch solutions.
Roughly speaking, the stability is the proposition claiming that the semilimit of a family of viscosity solutions is a viscosity solution;
see \cite[Theorem A.2]{BP87}.
At least in the classical theory of viscosity solutions the stability is a fundamental property to derive some asymptotic behavior of the solution.
Large time asymptotics of a viscosity solution of Hamilton-Jacobi equations is studied by Namah-Roquejoffre \cite{NR99} and Fathi \cite{F98} independently.
Another aim of this chapter is to establish, as a consequence of the stability, a large time asymptotic behavior of the solution on a singular space such as the Sierpinski gasket.
Based on the argument in \cite{NR99} we will show that the solution $u(t, x)+ct$ of \eqref{e:hj} goes to a function $v$ as $t \to \infty$ and $v$ is a solution of the stationary problem \eqref{e:lhj}
with some constant $c \in \mathbf{R}$.

We restrict ourselves to the case when the metric space $X$ is compact
but the Sierpinski gasket can be handled.
Let us extend $H$ to $X\times\mathbf{R}$ as an even function $H(x, p) = H(x, |p|)$.
The basic assumptions on the Hamiltonian are:
\begin{itemize}
\item[(A1)]
$H$ is continuous.
\item[(A2)]
$H$ is convex in the second variable.
\item[(A3)]
$H$ is coercive in the sense of
$$
\lim_{p \to \infty} \inf_x H(x, p) = \infty.
$$
\item[(A4)]
$\sup_x H(x, 0) < \infty$.
\end{itemize}
Set $c := \sup_x H(x, 0)$;
otherwise the stationary equation \eqref{e:lhj} has no solution.
We first show that the stationary equation has at least one solution.
Then, a standard barrier method implies that there exist upper and lower semilimits $\overline{u}(x)$ and $\underline{u}(x)$ of $u(t, x)+c t$ as $t \to \infty$ as real-valued functions.
As a result the stability argument yields that $\overline{u}$ and $\underline{u}$ are a subsolution and a supersolution of the limit equation.
Next note that for each $x \in A := \{ x \in X \mid H(x, 0) = \sup_x H(x, 0) \}$ the solution $u(t, x)+c t$ is non-increasing since $u_t+c \leq 0$
and so $\overline{u} = \underline{u}$ on $A$ by Dini's theorem.
We see that $\overline{u} = \underline{u}$ by a comparison principle for the stationary equation.
This means that $u(t, x)+c t$ converges to a solution $\overline{u} = \underline{u}$ of the stationary equation locally uniformly.

In order to justify this argument we will establish solvability of \eqref{e:lhj} and a comparison principle for Gangbo-\'{S}wi\c{e}ch solutions of \eqref{e:lhj},
which are new.
The authors of \cite{NR99} invoke a result by Lions-Papanicolaou-Varadhan \cite{LPV87}.
The argument is based on the ergodic theory
but in this work we will follow a direct approach via Perron method by Fathi-Siconolfi \cite{FS05}.

\section{Definition of Gangbo-\'{S}wi\c{e}ch solutions}
\label{s:def}

In this section we review the definition of metric viscosity solutions proposed by Gangbo and \'{S}wi\c{e}ch;
see \cite{GS14} and \cite{GS14b}.
Let $(X, d)$ be a complete geodesic metric space.

For a real-valued function $u$ on an open subset $Q$ of the spacetime $\mathbf{R}\times X$
define the \emph{upper local slope} and \emph{lower local slope}
\begin{align*}
|\nabla^+ u|(t, x) &:= \limsup_{y \to x} \frac{[u(t, y)-u(t, x)]_+}{d(y, x)}, \\
\quad |\nabla^- u|(t, x) &:= \limsup_{y \to x} \frac{[u(t, y)-u(t, x)]_-}{d(y, x)}
\end{align*}
and the \emph{local slope}
$$
|\nabla u|(t, x) := \limsup_{y \to x} \frac{|u(t, y)-u(t, x)|}{d(y, x)}.
$$
It is easy to see that $|\nabla^- u| = |\nabla^+ (-u)|$.

We next introduce smoothness classes for functions on a metric space.

\begin{definition}
We denote by $\mathcal{C}(Q)$ the set of all functions $u$ on $Q$ such that $u$ is locally Lipschitz continuous on $Q$ and $\partial_t u$ is continuous on $Q$.
We also set
$$
\begin{aligned}
\overline{\mathcal{C}}^1(Q) &:= \{ u \in \mathcal{C}(Q) \mid \text{$|\nabla^+ u| = |\nabla u|$ and they are continuous} \}, \\
\underline{\mathcal{C}}^1(Q) &:= \{ u \in \mathcal{C}(Q) \mid \text{$|\nabla^- u| = |\nabla u|$ and they are continuous} \}. \\
\end{aligned}
$$
\end{definition}

\begin{lemma}
\label{t:c1ex}
Let $u(t, x) := a(t)\phi(d(x, y)^2)+b(t)$ with $y \in X$, $\phi \in C^1(\mathbf{R}_+)$, $\phi' \geq 0$, $a, b \in C^1(\mathbf{R})$.
Then, $u \in \underline{\mathcal{C}}^1(\mathbf{R}\times X)$ and moreover
$$
|\nabla^- u|(t, x) = |\nabla u|(t, x) = 2a(t)\phi'(d(x, y)^2)d(x, y).
$$
\end{lemma}

See \cite{AF14} or \cite{GS14} for the proof.

We consider a Hamilton-Jacobi equation of the form
\begin{equation}
\label{e:ghj}
F(z, |\mathit{D}u|, \partial_t u) = 0 \quad \text{in $Q$.}
\end{equation}
Here, $z = (t, x)$, and $F = F(z, p, q) \in \mathit{C}(Q\times\mathbf{R}\times\mathbf{R})$ is even and convex in $p$ and strictly increasing in $q$.
Set
$$
F_r(z, p, q) :=
\begin{cases}
\sup_{p' \in B}F(z, p+r p', q) & \text{if $r \ge 0$} \\
\inf_{p' \in B}F(z, p+r p', q) & \text{if $r \le 0$} \\
\end{cases}
$$
for $r \in \mathbf{R}$,
where $B := [-1, 1]$.
Note that $(z, p, q, r) \mapsto F_r(z, p, q)$ is continuous
since $(z, p, q, r, p') \mapsto F(z, p+r p', q)$ is continuous and $B$ is compact.
Also, it is easy to check $r \mapsto F_r(z, p, q)$ is non-decreasing.

For a function $u$ defined on $Q$ with values in the extended real numbers $\bar{\mathbf{R}} := \mathbf{R}\cup\{ \pm\infty \}$,
we take its upper and lower semicontinuous envelope $u^*$ and $u_*$.

\begin{definition}[Metric viscosity solutions of \eqref{e:ghj}]
Let $u$ be an $\bar{\mathbf{R}}$-valued function on $Q$.

We say that $u$ is a \emph{metric viscosity subsolution} (resp.\ \emph{supersolution}) of \eqref{e:ghj}
when for every $\psi = \psi_1+\psi_2$ with $\psi_1 \in \underline{\mathcal{C}}^1(Q)$ (resp.\ $\psi_1 \in \overline{\mathcal{C}}^1(Q)$) and $\psi_2 \in \mathcal{C}(Q)$,
if $u^*-\psi$ (resp.\ $u_*-\psi$) attains a zero local maximum (resp.\ minimum) at a point $z = (t, x) \in Q$,
i.e.\ $(u^*-\psi)(z) = \max_{B_R(z)}(u^*-\psi) = 0$ (resp.\ $(u_*-\psi)(z) = \min_{B_R(z)}(u_*-\psi) = 0$) for some $R > 0$,
then
$$
\text{
$F_{-|\nabla \psi_2|^*(z)}(z, |\nabla \psi_1|(z), \partial_t\psi(z)) \le 0$
(resp.\ $F_{|\nabla \psi_2|^*(z)}(z, |\nabla \psi_1|(z), \partial_t\psi(z)) \ge 0$.)}
$$

We say that $u$ is a \emph{metric viscosity solution} of \eqref{e:ghj}
if $u$ is both a metric viscosity subsolution and a metric viscosity supersolution of \eqref{e:ghj}.
\end{definition}

By a similar way we also define a notion of metric viscosity solutions for a stationary equation of the form
\begin{equation}
\label{e:gshj}
H(x, |\mathit{D}v|) = 0 \quad \text{in $U$}
\end{equation}
with $U \subset X$ open.
Here, $H = H(x, p) \in \mathit{C}(U\times\mathbf{R})$ is even and convex in $p$.
Note that one is able to define the local slopes $|\nabla^- v|$, $|\nabla^+ v|$, $|\nabla v|$ and smoothness $\mathcal{C}(U)$, $\overline{\mathcal{C}}^1(U)$, $\underline{\mathcal{C}}^1(U)$ for a function $v$ on $U$.

\begin{definition}[Metric viscosity solutions of \eqref{e:gshj}]
Let $v$ be an $\bar{\mathbf{R}}$-valued function on $U$.

We say that $v$ is a \emph{metric viscosity subsolution} (resp.\ \emph{supersolution}) of \eqref{e:gshj}
when for every $\psi = \psi_1+\psi_2$ with $\psi_1 \in \underline{\mathcal{C}}^1(U)$ (resp.\ $\psi_1 \in \overline{\mathcal{C}}^1(U)$) and $\psi_2 \in \mathcal{C}(Q)$,
if $v^*-\psi$ (resp.\ $v_*-\psi$) attains a zero local maximum (resp.\ minimum) at a point $x \in U$,
then
$$
\text{
$H_{-|\nabla \psi_2|^*(x)}(x, |\nabla \psi_1|(x)) \le 0$
(resp.\ $H_{|\nabla \psi_2|^*(x)}(x, |\nabla \psi_1|(x)) \ge 0$.)}
$$

We say that $v$ is a \emph{metric viscosity solution} of \eqref{e:gshj}
if $v$ is both a metric viscosity subsolution and a metric viscosity supersolution of \eqref{e:gshj}.
\end{definition}

These notions satisfies the following natural propositions.

\begin{proposition}[Consistency]
\label{t:gshj}
If $u$ is a metric viscosity subsolution of \eqref{e:ghj} in $I\times U$ with an open interval $I$ and $u$ is of the form $u(t, x) = v(x)$,
then $v$ is a metric viscosity subsolution of \eqref{e:gshj} in $U$ with $H(x, p) := F(0, x, p, 0)$.

Conversely, if $v$ is a metric viscosity subsolution of \eqref{e:gshj} in $U$,
then $u(t, x) := v(x)$ is a metric viscosity subsolution of \eqref{e:ghj} in $\mathbf{R}\times U$ with $F(t, x, p, a) := H(x, p)$.
\end{proposition}

\begin{proposition}[Transitive relation]
\label{t:trans}
Assume that $F = F(z, p, q)$, $G = G(z, p, q)$ satisfy $G \leq F$
and let $u$ be a metric viscosity subsolution of \eqref{e:ghj}.
Then, $u$ is a metric viscosity subsolution of $G(z, |\mathit{D}u|, \partial_t u) = 0$ in $Q$.
\end{proposition}

\begin{proposition}[Locality]
\label{t:loc}
Let $Q_1$ and $Q_2$ be two open subsets of $(0, \infty)\times X$.
If $u$ is a metric viscosity subsolution of \eqref{e:ghj} in $Q_1$ and is a metric viscosity subsolution of \eqref{e:ghj} in $Q_2$,
then $u$ is a metric viscosity subsolution of \eqref{e:ghj} in $Q = Q_1\cup Q_2$
\end{proposition}

\begin{proposition}[Change of variable]
\label{t:cvt}
Let $\phi$ be a $C^1$ diffeomorphism from an interval $I$ to an interval $J$.
If $u$ is a metric viscosity subsolution of \eqref{e:ghj} in $I\times U$,
then $v(s, x) := u(\phi^{-1}(s), x)$ is a metric viscosity subsolution of
$$
F(x, |D v|, \phi' v_s) = 0 \quad \text{in $J\times U$.}
$$
\end{proposition}

\begin{proposition}[Composition]
\label{t:cvu}
Let $a$ be a non-zero constant and $b = b(t)$ be a $C^1$ function on an interval $I$.
If $u$ is a metric viscosity subsolution of \eqref{e:ghj} in $I\times U$,
then $v(t, x) := a u(t, x)+b(t)$ is a metric viscosity subsolution of
$$
F(x, \frac{|D v|}{a}, \frac{v_t-b'(t)}{a}) = 0 \quad \text{in $I\times U$.}
$$
\end{proposition}

\begin{proposition}[Strong solutions]
\label{t:ss}
Let $\psi_1 \in \overline{\mathcal{C}}^1(Q)$ and $\psi_2 \in \mathcal{C}(Q)$.
If $\psi = \psi_1+\psi_2$ satisfies
$$
F_{|\nabla\psi_2|^*(z)}(z, |\nabla\psi_1|(z), \partial_t\psi(z)) \leq 0 \quad \text{for all $z \in Q$,}
$$
then $\psi$ is a metric viscosity subsolution of \eqref{e:ghj}.
\end{proposition}

The proofs are straightforward
so we omit them.
For the proof of Proposition \ref{t:ss} see \cite[Lemma 2.8]{GS14b}.

\section{Stability results}
\label{s:stab}

Let $A$ be a topological space.
For a family of functions $\{u(\cdot; a)\}_{a \in A}$ defined on $Q$ take its upper and lower semicontinuous envelopes
$$
u^*(z; a) := \limsup_{(z', a') \to (z, a)}u(z'; a'),
\quad u_*(z; a) := \liminf_{(z', a') \to (z, a)}u(z'; a').
$$
The functions $u^*(\cdot; a)$ and $u^*(\cdot; a)$ are respectively called the upper and lower semilimit of $\{u(\cdot; a)\}$ at $a \in A$.
Also note that $u^*(\cdot; a)$ is upper semicontinuous and that for each $(z, a)$ there exists a sequence $(z_j, a_j)$ such that
$$
(z_j, a_j, u(z_j; a_j)) \to (z, a, u^*(z; a)).
$$

One of the main results of this section is:

\begin{lemma}[Stability]
\label{t:stab}
Let $F = F(z, p, q; a) \in \mathit{C}(Q\times\mathbf{R}\times\mathbf{R}\times A)$
and let $u = u(\cdot; a)$ be a family of metric viscosity subsolutions (resp.\ supersolution) of \eqref{e:ghj} with $F = F(\cdot; a)$.
Assume that $a \in A$ satisfies for each $z \in Q$
\begin{equation}
\label{e:astab}
\text{$\limsup_{a' \to a}\sup_{B_r(z)}u(\cdot; a') \leq \sup_{B_r(z)}u^*(\cdot; a)$
(resp.\ $\liminf_{a' \to a}\inf_{B_r(z)}u(\cdot; a') \leq \inf_{B_r(z)}u_*(\cdot; a)$)} \quad \text{}
\end{equation}
for all $r > 0$ small enough.
Then, the upper (resp.\ lower) semilimit $\overline{u} := u^*(\cdot, a)$ (resp.\ $\underline{u} := u^*(\cdot, a)$) is a metric viscosity subsolution (resp.\ supersolution) of \eqref{e:ghj} with $F = F(\cdot; a)$.
\end{lemma}

\begin{remark}
\label{t:scstab}
An sufficient condition of the assumption \eqref{e:astab} is that the metric space $X$ is locally compact.
Indeed, since $B := \overline{B}_r(z)$ is compact for small $r$,
we are able to take a sequence of maximum points $z_{a'}$ of $u(\cdot; a')$ and assume that $z_{a'}$ converges to some $\bar{z} \in B$ as $a' \to a$ by taking a subsequence if necessary.
Then,
$$
\limsup_{a' \to a}\sup_{B}u(\cdot; a') = \limsup_{a' \to a}u(z_{a'}; a') \le u^*(\bar{z}; a) \le \sup_{B}u^*(\cdot; a).
$$

We also point out that if the assumption is removed, then the lemma may be false in general;
see \cite{GHN14}.
\end{remark}

A direct consequence of Lemma \ref{t:stab} is:

\begin{corollary}[Stability under extremum]
\label{t:sup}
Let $F = F(z, p, q) \in \mathit{C}(Q\times\mathbf{R}\times\mathbf{R})$.
Let $S$ be a family of metric viscosity subsolutions (resp.\ supersolutions) of \eqref{e:ghj}.
Then $\overline{u}(z) := \sup_{v \in \mathcal{S}}v(z)$ (resp.\ $\underline{u}(z) := \inf_{v \in \mathcal{S}}v(z)$) is a metric viscosity subsolution (resp.\ supersolutions) of \eqref{e:ghj}.
\end{corollary}


\begin{proof}
Set $A = S$ with the indiscrete topology
and trivial families $\{F\}_{v \in S}$ and $\{U(\cdot; v) = v\}_{v \in S}$.
Note that $U^*(z; v) = u^*(z)$ and $\limsup_{v' \to v}\sup_{B_r(z)}U(\cdot; v') = \sup_{B_r(z)}U^*(\cdot; v) = \sup_{B_r(z)}u^*$.
Therefore, by applying Lemma \ref{t:stab} we see that $u$ is a metric viscosity subsolution of \eqref{e:ghj}.
\end{proof}

Our proof of Lemma \ref{t:stab} is inspired by \cite{GS14}.
First recall Ekeland's variational principle of a classical version \cite{E74}, \cite{E79}.

\begin{lemma}[Ekeland's variational principle]
\label{t:ekeland}
Let $(X, d)$ be a complete metric space and let $F \colon X \to \bar{\mathbf{R}}$ be a upper semicontinuous function bounded from above (resp.\ below) satisfying $D(F) := \{ F > -\infty \}$ (resp.\ $D(F) := \{ F < +\infty \}$) is not empty.
Then, for each $\hat{x} \in D(F)$, there exists $\bar{x} \in X$ such that $d(\hat{x}, \bar{x}) \le 1$, $F(\bar{x}) \ge F(\hat{x})$ (resp.\ $F(\bar{x}) \le F(\hat{x})$) and $x \to F(x)-m d(\bar{x}, x)$ attains a strict maximum (resp.\ minimum) at $\bar{x}$ with $m := \sup F-F(\hat{x})$ (resp.\ $m := \inf F-F(\hat{x})$).
\end{lemma}

See \cite{E79} for the proof.

\begin{proof}[Proof of Lemma \ref{t:stab}]
Fix $\psi = \psi_1+\psi_2$ with $\psi_1 \in \underline{\mathcal{C}}^1(Q)$ and $\psi_2 \in \mathcal{C}(Q)$ such that $\overline{u}-\psi$ attains a zero maximum at $\hat{z} = (\hat{t}, \hat{x})$ over $\overline{B}_R(\hat{z}) \subset Q$ with some $R > 0$, i.e.
\begin{equation}
\label{e:stabtest}
(\overline{u}-\psi)(\hat{z}) = \sup_{\overline{B}_R(\hat{z})}(\overline{u}-\psi) = 0.
\end{equation}
Set $\tilde{\psi}_2(z) := \psi_2(z)+d(\hat{x}, x)^2+(t-\hat{t})^2$ with $z = (t, x)$ and $\tilde{\psi} = \psi_1+\tilde{\psi}_2$.

Take a subsequence $a_j \to a$ and a sequence of points $z_j = (t_j, x_j) \in \overline{B}_R(\hat{z})$ such that $z_j \to \hat{z}$ and $u_j(z_j) = u(z_j; a_j) \to \overline{u}(\hat{z}) = u^*(\hat{z}; a)$,
where $u_j := u(\cdot; a_j)$.
We see by Ekeland's variational principle (Lemma \ref{t:ekeland}) that there exists $w_j = (s_j, y_j) \in \overline{B}_R(\hat{z})$ such that $z = (t, x) \mapsto ((u_j)^*-\tilde{\psi})(z)-m_j d(y_j, x)$ attains maximum at $w_j$ over $\overline{B}_R(\hat{z})$ with
$$
m_j := \sup_{\overline{B}_R(\hat{z})}((u_j)^*-\tilde{\psi})-((u_j)^*-\tilde{\psi})(z_j) \geq 0
$$
Note that
$$
\limsup_{j \to \infty}m_j = \limsup_{j \to \infty}\sup_{\overline{B}_R(\hat{z})}((u_j)^*-\tilde{\psi})-(\overline{u}-\tilde{\psi})(\hat{z})
$$
and so $m_j \to 0$ by the assumptions \eqref{e:astab} and \eqref{e:stabtest}.
We also observe that
$$
((u_j)^*-\tilde{\psi})(w_j)
\ge (u_j-\tilde{\psi})(z_j)-m_j d(y_j, x_j)
\ge (u_j-\tilde{\psi})(z_j)-2R m_j
$$
and that the last term converges to $(\overline{u}-\tilde{\psi})(\hat{z})$ as $j \to \infty$.
Therefore,
$$
\begin{aligned}
\limsup_{j \to \infty}d(\hat{x}, y_j)^2+(s_j-\hat{t})^2
&\le \limsup_{j \to \infty}((u_j)^*-\psi)(w_j)-(\overline{u}-\tilde{\psi})(\hat{z}) \\
&\leq \limsup_{j \to \infty}\sup_{\overline{B}_R(\hat{z})}((u_j)^*-\psi)-(\overline{u}-\tilde{\psi})(\hat{z})
\end{aligned}
$$
and it follows from \eqref{e:astab} and \eqref{e:stabtest} that $w_j = (s_j, y_j) \to \hat{z} = (\hat{t}, \hat{x})$.

Now, since $u_j$ is a metric viscosity subsolution,
$$
F_{-r_j}(w_j, |\nabla \psi_1|(w_j), \partial_t\tilde{\psi}(w_j); a_j) \leq 0.
$$
Here, $r_j$ is some non-negative number such that
$$
r_j \le |\nabla \psi_2|^*(w_j)+m_j+2d(\hat{x}, y_j)+2|s_j-\hat{t}|
$$
and so $\limsup r_j \le |\nabla \psi_2|^*(\hat{z})$.
Since $(z, p, q, r; a) \to F_r(z, p, q; a)$ is continuous and $r \to F_r(z, p, q; a)$ is non-decreasing,
we see that
$$
F_{-|\nabla \psi_2|^*(\hat{z})}(\hat{z}, |\nabla \psi_1|(\hat{z}), \partial_t\psi(\hat{z}); a) \leq 0.
$$
Therefore, $u$ is a subsolution.
\end{proof}

Another goal of this section is a principle to construct a metric viscosity solution by the Perron method.

\begin{proposition}[Perron method]
\label{t:perron}
Let $F = F(z, p, q) \in \mathit{C}(Q\times\mathbf{R}\times\mathbf{R})$
and let $g$ be an $\bar{\mathbf{R}}$-valued function on $\partial Q$.
Let $\mathcal{S}$ denote the set of all metric viscosity subsolutions (resp.\ supersolution) $v$ of \eqref{e:ghj} satisfying $v^* \le g$ (resp.\ $v_* \le g$) on $\partial Q$.
Then, $u(z) := \sup_{v \in \mathcal{S}}v(z)$ (resp.\ $u(z) := \inf_{v \in \mathcal{S}}v(z)$) is a metric viscosity solution of \eqref{e:ghj}.
\end{proposition}

Perron method for construction of a viscosity solution to Hamilton-Jacobi equations was first presented by H.~Ishii \cite{I87}.
Actually, the authors of \cite{GS14} have already established a similar result for metric viscosity solutions (\cite[Theorem 7.6]{GS14}).
However, let us give a proof since we have slightly improved the result to apply it directly to construction of a solution of the limit equation \eqref{e:lhj}.
We remark that the function
$$
\overline{u}(x) :=
\begin{cases}
+\infty & \text{if $x \in Q$} \\
g(x) & \text{if $x \in \partial Q$}
\end{cases}
$$
is a supersolution of \eqref{e:ghj}.

\begin{proof}
We only show that $u$ is a supersolution since being a subsolution is due to Corollary \ref{t:sup}.
Fix $\psi = \psi_1+\psi_2$ with $\psi_1 \in \overline{\mathcal{C}}^1(Q)$ and $\psi_2 \in \mathcal{C}(Q)$ such that $u_*-\psi$ attains a zero minimum at $\hat{z} := (\hat{t}, \hat{x})$ over $\overline{B}_R(\hat{z}) \subset Q$ with some $R > 0$.
Set $\tilde{\psi}_2(z) := \psi_2(z)-d(\hat{z}, z)^2$ and $\tilde{\psi} = \psi_1+\tilde{\psi}_2$.
Suppose by contradiction that
$$
F_{|\nabla \psi_2|^*(\hat{z})}(\hat{z}, |\nabla \psi_1|(\hat{z}), \partial_t\psi(\hat{z})) < 0.
$$
Since $(z, p, q, r) \mapsto F_r(z, p, q)$ is continuous and $r \mapsto F_r(z, p, q)$ is non-decreasing,
we may assume that $\tilde{\psi} = \psi_1+\tilde{\psi}_2$ is a subsolution of
$$
F_{|\nabla \tilde{\psi}_2|^*(z)}(z, |\nabla \psi_1|(z), \partial_t\tilde{\psi}(z)) \leq 0 \quad \text{for all $z \in B_R(\hat{z})$}
$$
by taking $R$ small enough.
Recalling Proposition \ref{t:ss}, we see that $\tilde{\psi}$ is a metric viscosity subsolution of \eqref{e:ghj} in $B_R(\hat{z})$.
Now observe that
$$
(u-\tilde{\psi})(z)
\geq (u_*-\tilde{\psi})(z)
\geq d(\hat{z}, z)^2 \geq \frac{R^2}{4} =: m > 0
$$
for all $z \in \overline{B}_{R}(\hat{z})\setminus B_{R/2}(\hat{z})$.
Construct a new function
$$
w(z) =
\begin{cases}
\max\{ \tilde{\psi}(z)+m/2, u(z) \} & \text{if $z \in B_{R}(\hat{z})$,} \\
u(z) & \text{otherwise.}
\end{cases}
$$
Then, $w$ is equal to $u$ on $Q\setminus B_{R/2}(\hat{z})$ and so it is a subsolution of \eqref{e:ghj} in $Q\setminus B_{R/2}(\hat{z})$.
It follows from Proposition \ref{t:cvu} and Corollary \ref{t:sup} that $w$ is a subsolution of \eqref{e:ghj} in $B_{R}(\hat{z})$.
Therefore, Proposition \ref{t:loc} shows that $w$ is a subsolution of \eqref{e:ghj} in $Q$ and so $w \in \mathcal{S}$.
In particular, $u(\hat{z}) \ge w(\hat{z})$ but $w(\hat{z}) = \psi(\hat{z})+m/2 = u(\hat{z})+m/2$.
Since $m > 0$, we obtain a contradiction and conclude that $u$ is a supersolution.
\end{proof}

\section{Application to large time behavior}
\label{s:longtime}

We study large time asymptotic behaviors of solutions of the Hamilton-Jacobi equation \eqref{e:hj} with a Hamiltonian $H$ satisfying (A1)--(A4) on a compact geodesic metric space $(X, d)$.
First note uniqueness of the constant $c$ such that \eqref{e:lhj} admits a solution.

\begin{proposition}
Assume (A1), (A2) and that $X$ is compact.
Let $c \in \mathbf{R}$ be a constant such that \eqref{e:lhj} admits a real-valued continuous solution.
Then,
\begin{equation}
\label{e:criv}
c = \max_{x \in X}H(x, 0).
\end{equation}
\end{proposition}

\begin{proof}
By the assumption (A2) it is enough to show that
\begin{equation}
\label{e:estcriv}
\sup_{x \in X}\inf_{p \in \mathbf{R}_+}H(x, p) \le c \le \sup_{x \in X}H(x, 0).
\end{equation}
It is easy to show the second inequality of \eqref{e:estcriv}.
Indeed, since a solution $v$ attains a minimum at some point $\hat{x} \in X$,
we have $H_0(\hat{x}, 0) = H(\hat{x}, 0) \ge c$, which implies $\sup_{x \in X}H(x, 0) \ge c$.
In order to prove the first inequality of \eqref{e:estcriv}, fix $\hat{x} \in X$.
Let us consider the function $v(x)-n d(x, \hat{x})^2/2$ and take its maximum point $x_n$ for each $n = 0, \cdots$.
Now, since $v(x_n)-n d(x_n, \hat{x})^2/2 \ge v(\hat{x})$, we have $d(x_n, \hat{x})^2 \le 2(\max v-v(\hat{x}))/n$.
Therefore, $x_n \to \hat{x}$ as $n \to \infty$.
Since $u$ is a subsolution, $H_0(x_n, n d(x_n, \hat{x})) = H(x_n, n d(x_n, \hat{x})) \le c$.
Hence, $\inf_{p \in \mathbf{R}_+}H(x_n, p) \le c$ and sending $n \to \infty$ yields $\inf_{p \in \mathbf{R}_+}H(\hat{x}, p) \le c$.
We now obtained the inequalities \eqref{e:estcriv}.
\end{proof}

\begin{remark}
It is a problem whether the inequalities \eqref{e:estcriv} holds even if we remove the compactness assumption.
One can show them by a similar argument to the proofs in Section \ref{s:stab} using Ekeland's variational principle provided $p \mapsto \sup_{x \in X}H(x, p)$ is continuous.
\end{remark}

In view of this proposition we hereafter define $c$ by \eqref{e:criv}.
Now we are able to state the main theorem of large time behavior.

\begin{theorem}[Large time behavior]
\label{t:ltb}
Assume (A1)--(A4), $u_0 \in \mathit{Lip}(X)$ and that $X$ is compact.
Let $u$ be a Lipschitz continuous solution of \eqref{e:hj} with \eqref{e:ic} on $[0, \infty)\times X$.
Then, $u(t, x)+c t$ converges to a function $v$ locally uniformly as $t \to \infty$ in X
and $v$ is a solution of \eqref{e:lhj}.
\end{theorem}

In order to prove this theorem we first establish regularity, existence and comparison results for the stationary equation \eqref{e:lhj}.
Set
$$
A := \{ x \in X \mid H(x, 0) = c \}.
$$

\begin{proposition}[Lipschitz continuity of solutions of \eqref{e:lhj}]
\label{t:liplhj}
Assume (A1), (A3) and (A4).
Then, real-valued continuous solutions of \eqref{e:lhj} are equi-Lipschitz continuous.
\end{proposition}

\begin{proof}
Note that there exists a constant $L \in \mathbf{R}_+$ such that $H(x, p) \ge c$ for all $x \in X$ and $p \ge L$ by (A3).
Fix a real-valued continuous solution $v$.
Consider the function $v(x)-v(y)-2L\sqrt{d(x,y)^2+\varepsilon^2}$ for $x, y \in X$ and take its maximum point $x_\varepsilon$ with respect to $x$ for each $\varepsilon > 0$.
Note that $x \mapsto v(y)+2L\sqrt{d(x, y)^2+\varepsilon^2}$ is of $\underline{\mathcal{C}}^1(X)$ by Lemma \ref{t:c1ex}.
Hence, we see that
$$
H_0\left(x_\varepsilon, \frac{2L d(x_\varepsilon, y)}{\sqrt{d(x_\varepsilon, y)^2+\varepsilon^2}}\right)
= H\left(x_\varepsilon, \frac{2L d(x_\varepsilon, y)}{\sqrt{d(x_\varepsilon, y)^2+\varepsilon^2}}\right) \le c.
$$
Therefore, we see that $2L d(x_\varepsilon, y)/\sqrt{d(x_\varepsilon, y)^2+\varepsilon^2} \le L$
and so $x_\varepsilon \to y$ as $\varepsilon \to 0$.
Now, for each $x, y \in X$, we have
$$
v(x)-v(y)-2L\sqrt{d(x,y)^2+\varepsilon^2}
\le v(x_\varepsilon)-v(y)-2L\sqrt{d(x_\varepsilon,y)^2+\varepsilon^2}
$$
Sending $\varepsilon \to 0$ yields
$$
v(x)-v(y)-2L d(x,y) \le 0,
$$
which means that all subsolutions of \eqref{e:lhj} is $2L$-Lipschitz continuous.
\end{proof}

\begin{theorem}[Existence of a solution of \eqref{e:lhj}]
\label{t:elhj}
Assume (A1), (A3) and (A4).
Then, there exists at least one Lipschitz continuous solution of \eqref{e:lhj} whenever $A$ is non-empty.
\end{theorem}

\begin{proof}
Define
$$
S(x, y) := \sup\{ w(x) \mid \text{$w \in \mathit{C}(X)$ is a subsolution of \eqref{e:lhj} with $w(y) = 0$} \}.
$$
Note that the constant $w \equiv 0$ is a subsolution of \eqref{e:lhj}.
Also Proposition \ref{t:liplhj} ensures that the solutions of \eqref{e:lhj} are equi-Lipschitz continuous and hence $v := S(\cdot, y)$ is a Lipschitz continuous function on $X$.
Now, Corollary \ref{t:sup} implies that $v$ is a subsolution of \eqref{e:lhj} in $X$
while Proposition \ref{t:perron} shows that $v$ is a supersolution of \eqref{e:lhj} in $X\setminus\{ y \}$.
Since $H(x, p) \ge H(x, 0) = c$ for $x \in A$,
we see that $v = S(\cdot, y)$ is a solution for every $y \in A \neq \emptyset$.
\end{proof}

\begin{theorem}[Comparison principle for \eqref{e:gshj}]
\label{t:scp}
Let $U$ be an open subset of $X$ such that $\overline{U}$ is compact.
Assume (A1), (A2) and that $H(x, 0) < 0$ for all $x \in U$.
Let $u$ be a subsolution and $v$ be a supersolution of \eqref{e:gshj} such that $u^* < +\infty$ and $v_* > -\infty$.
If $u^* \le v_*$ on $\partial U$,
then $u^* \le v_*$ in $U$.
\end{theorem}

\begin{proof}
First note that we may assume $u^*(x_0) \ne -\infty$ and $v_*(x_0) \ne +\infty$ at some $x_0 \in U$;
otherwise the conclusion holds.
Fix $\theta \in (0, 1)$ and consider the upper semicontinuous function defined by
$$
\Phi(x, y) := \theta u^*(x)-v_*(y)-\frac{1}{2\varepsilon}d(x, y)^2
$$
for $\varepsilon > 0$.
Thanks to the compactness of $\overline{U}$,
we are able to take a maximum point $(x_\varepsilon, y_\varepsilon) \in \overline{U}\times\overline{U}$ of $\Phi$.
It follows from $\Phi(x_\varepsilon, y_\varepsilon) \ge \Phi(x_0, x_0)$ that
$$
\begin{aligned}
\frac{1}{2\varepsilon}d(x_\varepsilon, y_\varepsilon)^2
&\leq \theta u^*(x_\varepsilon)-v_*(y_\varepsilon)-\theta u^*(x_0)+v_*(x_0) \\
&\leq \theta\sup u^*-\inf u_*-\theta u^*(x_0)+v_*(x_0) < +\infty.
\end{aligned}
$$
Hence, $d(x_\varepsilon, y_\varepsilon) \to 0$
and so we may assume that $x_\varepsilon$ and $y_\varepsilon$ converge to a same point $\bar{x} \in \overline{U}$ by taking a subsequence.
Let us consider the case when $\bar{x} \in U$.
Then, since $u$ and $v$ are a subsolution and a supersolution,
$$
\begin{aligned}
H(x_\varepsilon, \frac{1}{\theta\varepsilon}d(x_\varepsilon, y_\varepsilon)) &\le 0, \\
H(y_\varepsilon, \frac{1}{\varepsilon}d(x_\varepsilon, y_\varepsilon)) &\ge 0.
\end{aligned}
$$
By the convexity of $H$ the second inequality yields
$$
(1-\theta)H(y_\varepsilon, 0)+\theta H(y_\varepsilon, \frac{1}{\theta\varepsilon}d(x_\varepsilon, y_\varepsilon)) \ge 0
$$
Hence,
$$
(1-\theta)H(y_\varepsilon, 0)+\theta H(y_\varepsilon, \frac{1}{\theta\varepsilon}d(x_\varepsilon, y_\varepsilon))-\theta H(x_\varepsilon, \frac{1}{\theta\varepsilon}d(x_\varepsilon, y_\varepsilon)) \ge 0.
$$
Sending $\varepsilon \to 0$ yields $(1-\theta)H(\bar{x}, 0) \ge 0$.
Since $H(\bar{x}, 0) < 0$ and $\theta < 1$, we obtain a contradiction.
Therefore, $\bar{x} \in \partial U$.
We now observe that
$$
\theta u^*(x_\varepsilon)-v_*(y_\varepsilon) \ge \Phi(x_\varepsilon, y_\varepsilon) \ge \sup_{x \in U}\Phi(x, x) = \sup_U(\theta u^*-v_*).
$$
Hence, we see that $\sup_U(\theta u^*-v_*) \le (\theta u^*-v_*)(\bar{x}) \le \sup_{\partial U}(\theta u^*-v_*)$.
Sending $k \to 1$ implies $\sup_U(u^*-v_*) \le \sup_{\partial U}(u^*-v_*)$.
\end{proof}

\begin{corollary}[Comparison principle for \eqref{e:lhj}]
\label{t:lcp}
Assume that $X$ is compact.
Let $u$ be a subsolution and $v$ be a supersolution of \eqref{e:lhj} such that $u^* < +\infty$ and $v_* > -\infty$.
If $u^* \le v_*$ on $A$, then $u^* \le v_*$ on $X$.
\end{corollary}

\begin{proof}
It follows from the definition of $A$ that $H(x, 0)-c < 0$ for all $x \in U := X\setminus A$.
Therefore, Theorem \ref{t:scp} implies $u^* \le v_*$ in $X\setminus A$.
\end{proof}

We will also require a comparison principle for the evolution equation \eqref{e:hj}.


\begin{theorem}[Comparison principle for \eqref{e:hj}]
\label{t:cp}
Assume (A1) and that $X$ is compact.
Let $u$ be a subsolution and $v$ be a supersolution of \eqref{e:hj} such that $u^* < +\infty$ and $v_* > -\infty$.
If $u^*|_{t = 0} \le v_*|_{t = 0}$, then $u^* \le v_*$ on $(0, \infty)\times X$.
\end{theorem}

One is able to prove this theorem with the same idea as in \cite[Proof of Proposition 3.3]{GS14b}
and so we omit the proof.

Before starting the proof of Theorem \ref{t:ltb},
let us explain that the initial value problem \eqref{e:hj}, \eqref{e:ic} admits a unique Lipschitz continuous solution.
We will construct a solution by Perron method
while the uniqueness is a direct consequence of the comparison principle (Theorem \ref{t:cp}).
Let $\Lip[u_0]$ denote the Lipschitz constant of $u_0$
and set $K = \max_{x \in X}|H(x, \Lip[u_0])|$.
First note that $\overline{u}(t, x) := u_0(x)+K t$ and $\underline{u}(t, x) := u_0(x)-K t$ are a Lipschitz continuous supersolution and subsolution on $[0, \infty)\times X$, respectively.
We then can construct a continuous solution $u$ such that $\underline{u} \le \overline{u}$ by using Proposition \ref{t:perron} and Theorem \ref{t:cp}.
Take a constant $L \in \mathbf{R}_+$ such that $H(x, p) \ge c$ for all $x \in X$ and $p \ge L$.
Then, we see that $|u(t, x)-u(s, y)| \le K|t-s|+L d(x, y)$ by a similar argument to the proof of Proposition{t:liplhj}.
Actually, this is a standard argument and we refer the reader to \cite{GH13}.

We are now able to prove the main theorem stated at the top of this section.

\begin{proof}[Proof of Theorem \ref{t:ltb}]
Take the solution $v_0$ of \eqref{e:lhj} in Theorem \ref{t:elhj}.
Noting that $u_0$ and $v_0$ are bounded since $X$ is compact,
we are also able to see that $v_0-M \leq u_0 \leq v_0+M$ for some large $M > 0$.
Recall Propositions \ref{t:gshj} and \ref{t:cvu}, which imply that $v_0-c t\pm M$ are solutions of \eqref{e:hj}.
We then see by a comparison principle for \eqref{e:hj} (Theorem \ref{t:cp}) that $v_0-c t-M \leq u \leq v_0-c t+M$.
Thus, the upper and lower semi-limits
\begin{align*}
\overline{v}(x) &:= \sup_{(t_j, x_j) \to (\infty, x)}\limsup_j\{u(t_j, x_j)+c t_j\}, \\
\underline{v}(x) &:= \inf_{(t_j, x_j) \to (\infty, x)}\liminf_j\{u(t_j, x_j)+c t_j\}
\end{align*}
can be defined as a bounded function on $X$
since $v_0-M \leq \underline{v} \leq \overline{v} \leq v_0+M$.

We next note that Propositions \ref{t:cvt} and \ref{t:cvu} show the function
$$
w^\lambda(t, x) := u\left(\frac{t}{\lambda}, x\right)+c\frac{t}{\lambda}
$$
is a solution of
\begin{equation}\label{e:lambda}
\lambda\partial_t w^\lambda+H(x, |\mathit{D}w^\lambda|) = c \quad \text{in $(0, \infty)\times X$}
\end{equation}
for each $\lambda > 0$.
Since
\begin{align*}
\overline{v}(x) &= \sup_{(t_j, x_j, \lambda_j) \to (t, x, 0)}\limsup_j w^{\lambda_j}(t_j, x_j), \\
\underline{v}(x) &= \inf_{(t_j, x_j, \lambda_j) \to (t, x, 0)}\liminf_j w^{\lambda_j}(t_j, x_j)
\end{align*}
for all $t > 0$ and $x \in X$,
i.e.\ $\overline{v}$ and $\underline{v}$ are respectively nothing but the upper and lower semilimit of $w^\lambda$ as $\lambda \to 0$,
the stability result (Proposition \ref{t:stab}) and Proposition \ref{t:gshj} shows that $\overline{v}$ and $\underline{v}$ are a subsolution and a supersolution of \eqref{e:lhj}, respectively.

We next claim that $\overline{v} = \underline{v}$ on the set $A$.
Indeed, for each $x \in A$, $u(t, x)+c t$ converges to some $v(x)$ since $\partial_t u+c \leq 0$ and so it is a decreasing sequence.
We also obtain that $u$ is equi-Lipschitz continuous.
By connecting these two facts, we see that $\overline{v} \leq v \leq \underline{v}$ on $A$.

Finally, the comparison principle (Theorem \ref{t:lcp}) shows that $\overline{v} \leq \underline{v}$ on the whole space $X$.
Thus, we can conclude that $u(t, x)+c t$ converges to some function $v = \overline{v} = \underline{v}$ which is a solution of \eqref{e:lhj}.
\end{proof}

\begin{remark}
The convexity assumption (A2) is used only to guarantee a comparison principle holds for the stationary equation \eqref{e:lhj}.
It is possible to weaken the condition.
For instance, let us consider the specific Hamiltonian $H(x, p) = \sqrt{|p|}$,
which is not convex.
One easily see that the equation \eqref{e:lhj} is equivalent to $|D v| = c^2$.
Since a comparison principle for the convex Hamiltonian $|D v| = c^2$ implies a comparison principle $\sqrt{|D v|} = c$,
the same behavior of the solution must occur to the Hamiltonian $H(x, p) = \sqrt{|p|}$.
This scheme works for quasiconvex Hamiltonians $H(x, p) = h(|p|)+f(x)$ with $h \colon \mathbf{R}_+ \to \mathbf{R}$ such that $h(p)-\lambda p$ is non-decreasing for some $\lambda > 0$.
\end{remark}


Let us introduce the functions $\phi_-, \phi_\infty \in C(X)$ by
\begin{align*}
\phi_-(x) &:= \inf_{t \ge 0}(u(x,t)+c t), \\
\phi_\infty(x) &:= \min\{ S(x,y)+\phi_{-}(y) \mid y \in A \},
\end{align*}
where the function $S$ is defined in the proof of Theorem \ref{t:elhj}.

\begin{proposition}\label{t:sols}
Let $u\in C([0,\infty)\times X)$ be a solution of \eqref{e:hj}.
Then $\phi_{-}$ and $\phi_{\infty}$ are solutions of \eqref{e:lhj}.
\end{proposition}

\begin{proof}
Note that $u(x,t)+ct$ is supersolution of \eqref{e:lhj}. 
Thus, by Corollary \ref{t:sup} and Proposition \ref{t:perron}, which can be proved similarly for \eqref{e:lhj} as well, we immediately see that $\phi_{-}$ is solution of \eqref{e:lhj}.
Since $S$ is solution of \eqref{e:lhj} as we showed in the proof of Theorem \ref{t:elhj}, applying Corollary \ref{t:sup} and Proposition \ref{t:perron} again implies that $\phi_{\infty}$ is solution of \eqref{e:lhj}.
\end{proof}

\begin{theorem}[Asymptotic profile]
We have
\begin{equation}\label{e:profile}
\lim_{t\to\infty}(u(x,t)+ct) = \phi_{\infty}(x)\quad\text{for all $x \in X$}.
\end{equation}
\end{theorem}

\begin{proof}
Denote the left-hand side of \eqref{e:profile} by $u_{\infty}(x)$. 
Since
\begin{equation*}
\phi_{-}(x)\le u(x,t)+ct\quad\text{for all $(x,t)\in X\times[0,\infty)$,}
\end{equation*}
it follows that $\phi_{-}\le u_{\infty}$ in $X$.
It also can be seen that $\phi_{\infty}\le \phi_{-}$ on $\mathcal{A}$ and thus we have $\phi_{\infty}\le u_{\infty}$ on $\mathcal{A}$ from above two inequalities.
This fact and Theorem \ref{t:lcp} leads us to the relationship that $\phi_{\infty}\le u_{\infty}$ on $X$.

We shall show the other inequality.
In order to do, set $v(x,t):=\inf_{s\ge t}(u(x,s)+cs)$ and then we can easily see $v(x,0)=\phi_{-}(x)$.
Since $\phi_{\infty}=\phi_{-}$ on $A$, Corollary \ref{t:lcp} implies that $\phi_{\infty}\ge\phi_{-}=v(\cdot,0)$ on $X$.
It can be considered that $\phi_{\infty}$ as a solution of \eqref{e:hj} with $\phi_{\infty}\ge u_{0}$ and thus by Theorem \ref{t:cp} we obtain $v(x,t)\le \phi_{\infty}(x)$ for all $x\in X$ and $t\ge0$.
Sending $t\to\infty$ yields $u_{\infty}\le\phi_{\infty}$ on $X$.
\end{proof}


\providecommand{\bysame}{\leavevmode\hbox to3em{\hrulefill}\thinspace}
\providecommand{\MR}{\relax\ifhmode\unskip\space\fi MR }
\providecommand{\MRhref}[2]{%
  \href{http://www.ams.org/mathscinet-getitem?mr=#1}{#2}
}
\providecommand{\href}[2]{#2}

\end{document}